\documentclass[14pt, dvipsnames]{extarticle}
\usepackage[left=2cm, top=2cm, right=2cm, bottom=20mm]{geometry} 

\usepackage[T1]{fontenc}
\usepackage[cmintegrals,cmbraces]{newtxmath}
\usepackage[lining]{ebgaramond}
\usepackage{ebgaramond-maths}
\usepackage{euscript}
\usepackage{xcolor}
\usepackage{graphicx}
\usepackage[figurename=Fig.]{caption}

\usepackage[tracking = true, letterspace = 50]{microtype}
\SetTracking{encoding=T1, shape=sc}{50}
\SetTracking{encoding=T1}{50}

\usepackage{amsthm} 
\usepackage{hyperref} 
\usepackage[hyphenbreaks]{breakurl}

\theoremstyle{definition}

\newtheorem{coroll}{Corollary}

\newtheorem*{remark*}{Remark}

\newtheorem*{prop*}{Proposition}

\theoremstyle{definition}
\newtheorem{defi}{Definition}

\newtheorem{example}{\bf Example}

\theoremstyle{remark}
\newtheorem{remark}{\bf Remark}

\usepackage{hyphenat} 
\usepackage{microtype} 
\usepackage{tikz-cd}   

\usepackage{enumerate}

\usepackage{etoolbox}
\patchcmd{\thmhead}{(#3)}{#3}{}{}
\makeatletter
\renewenvironment{proof}[1][\proofname]{
  \par\pushQED{\qed}\normalfont
  \topsep6\p@\@plus6\p@\relax
  \trivlist\item[\hskip\labelsep\bfseries\itshape #1\@addpunct{.}]
}{\popQED\endtrivlist\@endpefalse}
\makeatother

\usepackage{etoolbox}
\patchcmd{\thmhead}{(#3)}{#3}{}{}

\let\phi\varphi

\DeclareMathOperator{\G}{\mathbb{G}}

\DeclareMathOperator{\inv}{\mathrm{inv}}

\DeclareMathOperator{\Inv}{\mathrm{Inv}}
\DeclareMathOperator{\Pp}{\mathcal{P}}
\DeclareMathOperator{\X}{\mathbb{X}}
\DeclareMathOperator{\Sym}{\mathrm{Sym}}
\DeclareMathOperator{\Xx}{\mathfrak{X}}
\DeclareMathOperator{\Free}{\mathrm{Free}}
\DeclareMathOperator{\R}{\mathbb{R}}
\DeclareMathOperator{\diag}{\mathrm{diag}}

\makeatletter
  \DeclareSymbolFont{ntxletters}{OML}{ntxmi}{m}{it}
  \SetSymbolFont{ntxletters}{bold}{OML}{ntxmi}{b}{it}
  \re@DeclareMathSymbol{\leftharpoonup}{\mathrel}{ntxletters}{"28}
  \re@DeclareMathSymbol{\leftharpoondown}{\mathrel}{ntxletters}{"29}
  \re@DeclareMathSymbol{\rightharpoonup}{\mathrel}{ntxletters}{"2A}
  \re@DeclareMathSymbol{\rightharpoondown}{\mathrel}{ntxletters}{"2B}
  \re@DeclareMathSymbol{\triangleleft}{\mathbin}{ntxletters}{"2F}
  \re@DeclareMathSymbol{\triangleright}{\mathbin}{ntxletters}{"2E}
  \re@DeclareMathSymbol{\partial}{\mathord}{ntxletters}{"40}
  \re@DeclareMathSymbol{\flat}{\mathord}{ntxletters}{"5B}
  \re@DeclareMathSymbol{\natural}{\mathord}{ntxletters}{"5C}
  \re@DeclareMathSymbol{\star}{\mathbin}{ntxletters}{"3F}
  \re@DeclareMathSymbol{\smile}{\mathrel}{ntxletters}{"5E}
  \re@DeclareMathSymbol{\frown}{\mathrel}{ntxletters}{"5F}
  \re@DeclareMathSymbol{\sharp}{\mathord}{ntxletters}{"5D}
  \re@DeclareMathAccent{\vec}{\mathord}{ntxletters}{"7E}
\makeatother

\renewcommand{\epsilon}{\varepsilon}

\DeclareFontFamily{U}{EBSUB}{}
\DeclareFontShape{U}{EBSUB}{m}{it}{<-> EBGaramond-Italic-tlf-lgr}{}
\DeclareFontFamily{U}{EBSUB}{}
\DeclareSymbolFont{EBSUB}{U}{EBSUB}{m}{it}

\DeclareMathSymbol{\mu}{\mathord}{EBSUB}{`m}
\DeclareMathSymbol{\varDelta}{\mathord}{EBSUB}{`D}
\DeclareMathSymbol{\varOmega}{\mathord}{EBSUB}{`W}

\DeclareSymbolFont{ntxletters}{OML}{ntxmi}{m}{it} 
\DeclareMathSymbol{\lhook}{\mathrel}{ntxletters}{"2C}
\usepackage{fonttable}

 \hypersetup{
    colorlinks,
    linkcolor={red!50!black},
    citecolor={blue!50!black},
    urlcolor={blue!80!black}
}

\usepackage{stackrel}
\usepackage[all]{xy}

\usepackage{bm}

\usepackage{indentfirst}

\newcommand\unmarkedfootnote[1]{%
  \begingroup
  \renewcommand\thefootnote{}\renewcommand\footnotemark{}
  \footnote{#1}%
  \endgroup
}

\usepackage{tocloft}

\title{\bf On the Classification of $\bm{n}$-Valued\\ Monoids and Groups of Order 3\unmarkedfootnote{The work was funded within the framework of the HSE University Basic Research Program.}}

\date{}

\author{Mikhail Kornev}

\begin{document}

\maketitle

\begin{abstract}
The article presents a complete classification of $n$-valued monoids and groups of order 3. Important corollaries of this result are discussed.
\end{abstract}

\tableofcontents

\section{Introduction}\label{sec:intro}\label{intro}

In 1971, V.\,M. Buchstaber and S.\,P. Novikov proposed a construction motivated by the theory of characteristic classes \cite{Buchstaber_Novikov}. This construction describes a multiplication in which the product of any pair of elements is a multiset of $n$ points. An axiomatic definition of $n$-valued groups and the first results of their algebraic theory were obtained in a subsequent series of works by V.\,M. Buchstaber. Currently, the theory of $n$-valued (formal, finite, discrete, topological, and algebro-geometric) groups and their applications in various areas of mathematics and mathematical physics are being developed by a number of authors; see \cite{Buchstaber75, Kholodov81, Buchstaber90, Buchstaber_Veselov, Vershik, BuchRees, Buchstaber, BuchVesGaif, BuchVesnin, Kontsevich_type_polynomials} and the references therein. A comparative analysis of the theories of 1-valued and $n$-valued groups can be found in \cite{Borovik}.

We refer the reader to \cite{Buchstaber} for the foundational definitions and constructions in the theory of $n$-valued groups.

We introduce the notion of a non-reversible $n$-valued group (Definition \ref{reversible}). It turns out that, up to now, no explicit constructions of non-reversible groups have appeared in the existing literature. As noted by V. Buchstaber, an incorrect example of a non-reversible 2-valued group was given in \cite[page 246]{Borovik}. In the present note, the author constructs a family of non-reversible $n$-valued groups of order 3 for all $n \geqslant 3$ (Example \ref{nonreversible_example}).

A natural and well-known question is the classification of $n$-valued monoids and groups of order 3 up to isomorphism. In \cite{Vershik}, the notion of involutive $n$-valued groups was introduced (we refer to them as $\bigstar$-involutive, see Definition 6). It was shown that such groups can be identified with certain combinatorial algebras. A specific series (\ref{4k+3}) of $\bigstar$-involutive groups on three points was also constructed. The complete classification of all coset $\bigstar$-involutive $n$-valued groups of order 3 was achieved in \cite{Ponomarenko_24, Ponomarenko_Vasiliev_24}, based on the classification of all single-valued finite permutation groups of rank 3 (see references therein).

The main result of this work is the complete classification of $n$-valued monoids and groups of order 3 (see Proposition from Section \ref{class_section}), along with important consequences arising from a comparison with the results of \cite{Ponomarenko_Vasiliev_24}. This leads in Section \ref{props} to several conclusions concerning the relationships among the currently known classes of $n$-valued groups. Corollary \ref{1-monoids} reproduces the classical result \cite{OEIS, monoids_classification} that there are exactly 7 isomorphism classes of monoids of order 3. Corollary \ref{cor_1} says that there exist only two families of $n$-valued groups of order 3. According to Corollary  \ref{cor_2}, every noncommutative $n$-valued monoid of order 3 is isomorphic to one of the two noncommutative single-valued monoids with multiplication tables:
\begin{equation}\label{1-non-comm-monoids}
\begin{aligned}
1\ast 1 &= 1, &\qquad 1\ast 1 &= 1,\\
1\ast 2 &= {\textcolor{blue}1}, &\qquad 1\ast 2 &= {\textcolor{blue}2},\\
2\ast 1 &= {\textcolor{blue}2}, &\qquad 2\ast 1 &= {\textcolor{blue}1},\\
2\ast 2 &= 2. &\qquad 2\ast 2 &= 2.
\end{aligned}
\end{equation}By Corollary \ref{cor_3}, there exists only one family of non-reversible $n$-valued groups of order 3. In Example \ref{nonreversible_example}, an explicit family of nonreversible groups $\mathbb{X}_n$ of order $3$ is presented. Corollary \ref{cor_4} completes the classification of all $\bigstar$-involutive groups of order $3$. In \cite{BuchVesGaif}, in connection with Conway’s topograph, the notion of an involutive 2-valued group was introduced. We generalize this notion to arbitrary $n$ in Definition \ref{involutive_group}. Corollary \ref{cor_5} provides the classification of all such groups of order 3.

Thus, the results of the present work yield the rightmost inclusion in the following diagram of all possible strict embeddings among $n$-valued groups of arbitrary order: 
\begin{equation}\label{inclusion_diagram}
\begin{tikzcd}[row sep=2em, column sep=3em, ampersand replacement=\&]
\begin{aligned}\text{Commutative}\\ \text{Involutive}\end{aligned}
\arrow[r, hook]
\& \bigstar\text{-Involutive}
\arrow[r, hook]
\& \text{Reversible} \\[1em]
\& \text{Coset}
\arrow[u, hook]
\end{tikzcd}
\end{equation}

The author gratefully acknowledges Victor M. Buchstaber for posing the problems and for helpful discussions, and Andrey V. Vasil’ev for his interest in this work and for valuable comments.

\section{Classification}\label{class_section}

We introduce the following:

\begin{defi}
An {\it $n$-valued monoid} is a set $X$ equipped with an operation
$$\ast: X\times X\to \Sym^n(X),$$
where $\Sym^n(X) = X^{\times n}/S_n$ is the set of unordered $n$-tuples of elements of $X$. This operation must satisfy all axioms of an $n$-valued group from \cite{Buchstaber}, except for the axiom asserting the existence of an inverse map, that is:
\begin{itemize}
\item {\it Associativity.} The $n^2$-multisets
\begin{gather*}[x*w\mid w\in y \ast z],\\
[w\ast z\mid w\in x*y]
\end{gather*} 
coincide.

\item {\it Unit.} There exists an element $e\in X$ such that $$e\ast x = x\ast e = [x, x, …, x]$$ for every $x\in X$.
\end{itemize}
\end{defi}

We say that the order of an $n$-valued monoid $X$ is $k$ if $|X| = k$. Throughout, we identify the set $\Sym^n(X)$ with a subset of the ordinary free commutative monoid $\Free(X)$ generated by $X$. In this way, we regard the $n$-valued group $X$ as an $n$-valued analogue of a monoid algebra.

Consider the set $X = \{e, x_1, x_2, …, x_{k-1}\}$ (we assume $x_0 = e$). Suppose $X$ is equipped with the structure of an $n$-valued monoid. Then the corresponding multiplication table can be conveniently written as:
\begin{equation*}
\Xx^{\circledast 2} = A\Xx,
\end{equation*}
where $\Xx$ denotes the column vector $(e, x_1, x_2, …, x_{k-1})^{T}$, and $A$ is an $k^2\times k$ matrix whose rows record the multiplicities of the elements $e, x_1, …, x_{k-1}$ in the product $x_i\ast x_j$ for $i, j = 0, …, k-1$. Moreover, in each row of the matrix $A$, the sum of all entries equals $n$. Finally, $\Xx^{\circledast 2} = \Xx\circledast \Xx$ denotes the Kronecker product of the vector $\Xx$ with itself, taken with respect to the operation $\ast$. For example,
$$
\begin{pmatrix}x_1\\ x_2\end{pmatrix}^{\circledast 2} =
\begin{pmatrix}
x_1\ast x_1\\
x_1\ast x_2\\
x_2\ast x_1\\
x_2\ast x_2
\end{pmatrix}.
$$

The associativity conditions then take the form:
\begin{equation}\label{associativity_conditions}
(\Xx\circledast \Xx)\circledast\Xx = \Xx\circledast(\Xx\circledast \Xx).
\end{equation}

We will need the following:

\begin{defi}\label{monoid_isomorphisms}
An $n$-valued monoid $X$ is said to be {\it isomorphic} to an $n’$-valued monoid $X’$ if there exists a bijection $\phi: X\to X’$ preserving the unit, $\phi(e) = e’$, such that
$$
\frac{m^z_{x, y}}{n} = \frac{m^{\phi(z)}_{\phi(x), \phi(y)}}{n’},
$$
where $m^z_{x, y}$ denotes the multiplicity of the element $z$ in the product $x\ast y$.
\end{defi}

Recall from \cite[Lemma 1]{Buchstaber} that for any natural number $m$, one can construct an $mn$-valued monoid $G’$ on the same set $X$ from an $n$-valued monoid $G$ with multiplication $\mu$, via the diagonal map $\diag_m(x) = mx$:
\begin{equation}\label{diagonal}
\mu’: X \times X \xrightarrow{\mu} \Sym^n X \xrightarrow{\diag_m} \Sym^{mn}{X}.
\end{equation}

We are now ready to state one of the main results of this paper:

\begin{prop*}\label{class}
Let $X = \{e, x_1, x_2\}$ and
\begin{equation}\label{vector_equality}
\begin{pmatrix}x_1\\ x_2\end{pmatrix}^{\circledast 2}  =
\begin{pmatrix}
a_0 & a_1 & a_2\\
b_0 & b_1 & b_2\\
d_0 & d_1 & d_2\\
c_0 & c_1 & c_2
\end{pmatrix}\Xx = B\Xx.
\end{equation}
Then every $n$-valued monoid on the set $X$ is represented by one of the following matrices $B = B_j$, provided that all fractional expressions yield non-negative integers:

$$
B_1 = \begin{pmatrix}
a_0 & n - a_0 & 0 \\
0 & 0 & n \\
0 & 0 & n \\
c_0 & \frac{c_0 n}{a_0} & n - c_0 -  \frac{c_0 n}{a_0} \\
\end{pmatrix},
$$

$$B_2 = \begin{pmatrix} a_0 & a_1 & n - a_0 - a_1 \\ b_0 & b_1 & n - b_0 - b_1 \\ b_0 & b_1 & n - b_0 - b_1\\ \frac{b_0(n-b_0-b_1)+a_0b_1-a_1b_0}{n-a_0-a_1} & \frac{(b_0+b_1)(n-b_1)}{n-a_0-a_1} & \frac{(n- b_0-b_1)^2 + n(b_1-a_0-a_1) + a_1b_0 - a_0b_1}{n-a_0-a_1} \end{pmatrix},$$

$$B_3 = \begin{pmatrix} 0 & n & 0 \\ 0 & n & 0\\ 0 & n & 0\\ c_0 & c_1 & n-c_0-c_1 \end{pmatrix}, \ B_4 = \begin{pmatrix} 0 & n & 0\\ 0 & 0 & n\\ 0 & 0 & n\\ 0 & c_1 & n-c_1 \end{pmatrix},$$

$$B_5 = \begin{pmatrix} 0 & n & 0\\ 0 & n & 0\\ 0 & 0 & n \\ 0 & 0 & n\end{pmatrix}, \ B_6 = \begin{pmatrix} 0 & n & 0\\ 0 & 0 & n\\ 0 & n & 0\\ 0 & 0 & n \end{pmatrix}.$$

Moreover, two monoids corresponding to the series $B_i$ and $B_j$ ($i$ and $j$ not necessarily distinct) are isomorphic if and only if the matrix $B_i$ is proportional to the matrix $B_j$, or to the matrix $\phi(B_j)$, where $\phi: X \to X$ is an isomorphism of monoids such that $\phi(x_1) = x_2$.

\end{prop*}

\begin{proof}
The associativity conditions (\ref{associativity_conditions}) define in the parameter space $\R^{12}$ an algebraic set given by the intersection of 81 quadrics.

This set is contained in the solutions of the following system of 18 equations:
\begin{equation}\label{non-full_associativity_conditions}
(x_i\ast x_j)\ast x_k = x_i\ast(x_j\ast x_k),
\end{equation}
where the indices $i, j, k \in \{1,2\}$ are not all equal. For instance, the three equations corresponding to the associativity condition $(x_1\ast x_1)\ast x_2 = x_1\ast (x_1\ast x_2)$ take the form:
\begin{equation*}
\left\{
\begin{aligned}
&a_1 b_0 - a_0 b_1 - b_0 b_2 + a_2 c_0 &= 0,\\
&a_2 c_1 - b_1 b_2 - n b_0 &= 0,\\
&a_1 b_2 -a_2 b_1- b_2^2 + a_2 c_2 + n a_0 &= 0,\\
 \end{aligned}
\right.
\end{equation*}

Solving the system (\ref{non-full_associativity_conditions}) of quadratic equations over the reals, subject to the conditions $n \neq 0$ and the requirement that the sum of entries in each row equals $n$ (four more linear equations), using the {\tt Reduce} function in {\tt Wolfram Mathematica}, yields a list of 26 families of 12-dimensional vectors. By analyzing this list and restricting to non-negative integer solutions, we obtain the 6 series of matrices listed in the proposition.

As is easy to check, the associativity conditions for cubes $(x_i\ast x_i)\ast x_i = x_i\ast (x_i\ast x_i)$ and for expressions involving the unit hold for all these series: for commutative cases this is automatic, and for the two noncommutative series it can be directly verified. Hence, systems (\ref{associativity_conditions}) and (\ref{non-full_associativity_conditions}) are equivalent, and we obtain a complete list of $n$-valued monoids on three elements.

The isomorphism statements follow directly from the structure of the series and from the fact that every isomorphism of $n$-valued monoids on $X$ is given either by the identity or by the map $\phi(x_1) = x_2$.

\end{proof}

\section{Corollaries}\label{props}

First of all, we note that the result of the Proposition agrees with the known classification of single-valued monoids of order 3 \cite{OEIS, monoids_classification}:

\begin{coroll}\label{1-monoids}
There are exactly seven isomorphism classes of single-valued monoids of order 3, given by the matrices $B_j$:
$$\begin{pmatrix} 1& 0 & 0\\ 0& 0& 1\\ 0& 0& 1\\ 0& 0& 1 \end{pmatrix}\in B_1,\quad
\begin{pmatrix} 0& 0& 1\\ 1& 0& 0\\ 1& 0& 0\\ 0& 1& 0 \end{pmatrix}\in B_2,\quad
\begin{pmatrix} 0& 0& 1\\ 0& 0& 1\\ 0& 0& 1\\ 0& 0& 1 \end{pmatrix}\in B_2,$$
$$\begin{pmatrix} 0& 1& 0\\ 0& 0& 1\\ 0& 0 & 1\\ 0& 1& 0 \end{pmatrix}\in B_4,\quad
\begin{pmatrix} 0& 1& 0\\ 0& 0& 1\\ 0& 0& 1\\ 0& 0& 1 \end{pmatrix}\in B_4,\quad
\begin{pmatrix} 0& 1& 0\\ 0& 1& 0\\ 0& 0& 1\\ 0& 0& 1 \end{pmatrix}\in B_5,\quad
\begin{pmatrix} 0& 1& 0\\ 0& 0& 1\\ 0& 1& 0\\ 0& 0& 1 \end{pmatrix}\in B_6.$$
\end{coroll}

To fix notation, we recall the following from \cite[Section 2]{Buchstaber}:

\begin{defi}
An {\it $n$-valued group} $X$ is an $n$-valued monoid equipped with an inverse map $\inv: X\to X$, that is, a map such that for every $x\in X$,
\begin{equation}\label{inv_condition}x\ast\inv(x)\ni e,\quad \inv(x)\ast x\ni e.\end{equation}
\end{defi}

\begin{coroll}\label{cor_1}
Every $n$-valued group of order 3 is given by a matrix of type $B_1$ or $B_2$. In the first case, one has
$$a_0 c_0 > 0,$$
while in the second case,
$$b_0 + a_0(b_0(n-b_0-b_1)+a_0b_1-a_1b_0) > 0.$$
\end{coroll}

\begin{proof}
The statement follows from examining the first columns of the matrices $B_j$ in the classification of $n$-valued monoids of order 3.

\end{proof}

\begin{coroll}\label{cor_2}
Every commutative $n$-valued monoid of order 3 belongs to one of the first four series. Every noncommutative monoid is a diagonal {\normalfont(\ref{diagonal})} of one of the two noncommutative single-valued monoids {\normalfont(\ref{1-non-comm-monoids})}. In particular, all $n$-valued groups of order 3 are commutative, and there exist only two noncommutative multivalued monoids up to isomorphism.
\end{coroll}

Let us note that the commutativity of $n$-valued groups of order 3 was first established in \cite[Lemma 10]{Yagodovskii}. That work also introduced an approach to the classification of $n$-valued groups in $\R^3$ using the theory of algebraic varieties and linear deformations of group $n$-algebras. However, no explicit classification of all $n$-valued groups of order 3 was obtained there.

Following V. Buchstaber (private communication), for every $n$-valued group on a set $X$, in addition to the inverse map $\inv: X \to X$, we define a multivalued inverse map $\Inv: X \to \Pp(X)$:

\begin{defi}
The map $\Inv: X\to \Pp(X)$ takes values in the set of all finite subsets of $X$. It is defined by
$$\Inv(x) = \{ y\in X \mid x\ast y \ni e,\quad y\ast x\ni e \}.$$
\end{defi}

\begin{defi}\label{reversible}
An $n$-valued group has a uniquely defined inverse map if and only if $\inv = \Inv$. We call such groups {\it reversible}; otherwise, we call them {\it non-reversible}.
\end{defi}

\begin{example}
Clearly, every coset group is reversible.
\end{example}

It was observed by V. Buchstaber that the example given in \cite[page 246]{Borovik} as a non-reversible 2-valued group is, in fact, incorrect. It arises as the reduction modulo $n$ of the 2-valued group $\G_2$ (see, e. g.,  \cite[Example 2]{BK}). Namely, consider the set of (non-negative) residues $X = {0, 1, …, n-1}$ modulo $n$ with the operation
$$x\ast y = [x + y\ \mathrm{mod}\  n,\ |x - y|\ \mathrm{mod}\  n].$$
This set $X$ with such an operation does not form a 2-valued group, since associativity fails:
\begin{align*}
(1\ast 1)\ast (n - 1) &= [0, 2]\ast (n - 1) \\
&= [n-1, n-1, 1, n-3] \\
&\neq [1, 1, n-1, n-3] \\
&= 1\ast[0, n-2] \\
&= 1\ast(1\ast(n - 1)).
\end{align*}

This observation leads naturally to the question of whether non-reversible $n$-valued groups actually exist.

\begin{example}
It is easy to see that every $n$-valued group $\G$ on the singleton set $\{0\}$ is a diagonal construction \cite[Lemma 1]{Buchstaber} of the trivial group, and is therefore reversible. All $n$-valued groups $\G$ on the two-element set $\{0, 1\}$ with unit $0$ are also reversible: each such $n$-valued group is determined by a pair of non-negative integers $(a, b)$ with $a + b = n$, $a > 0$, since there is only one nontrivial relation: $1\ast 1 = a\cdot 0 + b\cdot 1$.
\end{example}

\begin{example}\label{nonreversible_example}
For each $n \geqslant 3$, there exists a non-reversible commutative cyclic \cite[Definition 1.3]{BuchVesnin} $n$-valued group $\X_n$ on the set $\{0, 1, 2\}$ (with unit $0$), in which the inverse map $\inv$ may be defined in four different ways: $\Inv(1) = \Inv(2) = \{1, 2\}$. The multiplication in this group is given by:
$$
\begin{aligned}
1 \ast 1 &= 0 + (n-1)\cdot 2, \\
1 \ast 2 &= 0 + 1 + (n-2)\cdot 2, \\
2 \ast 2 &= 0 + 2\cdot 1 + (n-3)\cdot 2.
\end{aligned}
$$
It is clear that $\mathbb{X}_n$ belongs to the series $B_2$. The group $\mathbb{X}_n$ is non-reversible, therefore it is noncoset. The first known examples of (other) non-coset groups were constructed in \cite{Vershik} using the technique of combinatorial algebras and Sylow-type theorems from finite group theory.
\end{example}

\begin{coroll}\label{cor_3}
An $n$-valued group of order 3 is non-reversible if and only if it belongs to the series $B_2$ with
$$b_0(a_0 + b_0(n - b_0 - b_1) + a_0b_1 - a_1b_0) > 0.$$
\end{coroll}

Let us reproduce here the following \cite[Definition 1.2]{Vershik} in order to fix notation:

\begin{defi}\label{bigstar_inv}
An $n$-valued group $(X, \ast, \inv)$ with multiplication $\ast$ and inverse map $\inv$ is called {\it $\bigstar$-involutive} if the inverse map $x\mapsto \inv(x)$ is an involution satisfying the following conditions:
\begin{enumerate}[\bf (i)]
\item For any $x\in X$, one has $x\ast y\ni e$ if and only if $y = \inv(x)$.

\item $m(x) = m(\inv(x))$, where $m(x) = m^{e}_{x, \inv(x)}$ is the multiplicity of $e$ in the product $x\ast\inv(x)$.

\item $\inv(x\ast y)= \inv(y)\ast\inv(x)$ for all $x, y\in X$.
\end{enumerate}
\end{defi}

\begin{remark*}
In \cite{Vershik, Buchstaber, Ponomarenko_Vasiliev_24}, $\bigstar$-involutive groups were simply called involutive. The notion of $\bigstar$-involutivity is motivated by $C^{*}$-algebras from functional analysis and noncommutative geometry. Let us note that in \cite{Buchstaber}, immediately after the definition in Section 11, there is an unjustified assertion that condition {\bf (iii)} follows from condition {\bf (i)}.
\end{remark*}

\begin{example}
Every $n$-valued coset group is $\bigstar$-involutive \cite[page 13]{Vershik}.
\end{example}

From condition {\bf (i)} in Definition \ref{bigstar_inv}, it follows that every $\bigstar$-involutive group is reversible. 

\begin{coroll} \label{cor_4}
Every $\bigstar$\hyp{}involutive $n$\hyp{}valued group has as its matrix $B$ in {\normalfont (\ref{vector_equality})} one of the following:
$$\widetilde{B}_1 = \begin{pmatrix}
m(x_1) & a(x_1) & n - m(x_1) - a(x_1) \\
0 & a(x_1, x_2) & n - a(x_1, x_2) \\
0 & a(x_1, x_2) & n - a(x_1, x_2)\\
m(x_2) & a(x_2) & n - m(x_2) - a(x_2)
\end{pmatrix},$$
or
$$\widetilde{B}_2 = \begin{pmatrix}
0 & a(x_1) & n - a(x_1)\\
n - 2a(x_1) & a(x_1) & a(x_1) \\
n - 2a(x_1) & a(x_1) & a(x_1)\\
0 & n - a(x_1) & a(x_1)
\end{pmatrix},$$
where
$$a(x_1, x_2) = r(n - m(x_1) - a(x_1)),$$
$$a(x_2) = r(n - a(x_1, x_2)),$$
$$r = \frac{m(x_2)}{m(x_1)},$$
and $m(x_1)$, $m(x_2)$, $a(x_1)$ are free parameters.
\end{coroll}

\begin{proof}
Indeed, according to \cite[Theorem 2.1]{Ponomarenko_Vasiliev_24}, every $\bigstar$\hyp{}involutive $n$\hyp{}valued group of order three may have, as the matrix $B$ in (\ref{vector_equality}), either the matrix $\widetilde{B}_1$ or $\widetilde{B}_2$. The matrices $\widetilde{B}_2$ belong to the series $B_2$. The matrix $\widetilde{B}_1$ belongs to the series $B_1$ when $m(x_1) + a(x_1) = n$, and belongs to the series $B_2$ otherwise.

\end{proof}

\begin{remark*}
In \cite[Theorem 2.1]{Ponomarenko_Vasiliev_24} only explicit necessary conditions (in our notation) on the entries of the matrix $B = (x_i\ast x_j)_{ij}$ are stated; it is not claimed that the corresponding $n$-valued groups arise for all possible values of the parameters $m(x_1)$, $m(x_2)$, and $a(x)$.
\end{remark*}

As shown in \cite[Theorem 1.1]{Ponomarenko_Vasiliev_24}, an $n$-valued group from the series $\widetilde{B}_2$ with $n = 2k + 1$ is a coset group if and only if
\begin{equation}\label{4k+3}
\widetilde{B}_2 =
\begin{pmatrix}
0 & k & k+1 \\
1 & k & k \\
1 & k & k \\
0 & k+1 & k
\end{pmatrix},
\end{equation}
and $4k + 3$ is a power of a prime number. The classification of coset groups from the series $\widetilde{B}_1$ is related to the classification of finite single-valued permutation groups of rank 3 and the theory of strongly regular graphs \cite{Ponomarenko_Vasiliev_24}.

In the paper \cite{BuchVesGaif}, in connection with Conway’s topograph, the notion of an involutive 2-valued group was introduced.

\begin{defi}\label{involutive_group}
A reversible $n$-valued group on a set $X$ is called {\it involutive} if $\inv(x) = x$ for every $x\in X$.
\end{defi}

Clearly, every commutative involutive group is $\bigstar$-involutive. The article \cite{BuchVesGaif} gives a complete classification of involutive 2-valued groups. Let us note that in that paper, in Remark 1.5, there is an inaccuracy: involutivity of an $n$-valued group in the sense of \cite{Vershik} is not limited to the requirement of reversibility and involutivity of the map $\inv$, since two further conditions must also be satisfied (see {\bf (ii)} and {\bf (iii)} in Definition \ref{bigstar_inv} of the present paper). The statement given there (in our terminology) that the class of $\bigstar$-involutive groups contains the class of involutive groups is by no means obvious and follows for $n = 2$ from the (later) work \cite{Gaifullin_24}, in which it is shown that every involutive $2$-valued group is commutative. Examples of involutive non-coset groups can be found in \cite[Proposition 4.4]{BuchVesGaif}.

\begin{coroll}\label{cor_5}
Groups from the series $B_1$ are involutive. A group from the series $B_2$ is involutive if and only if $b_0 = 0$.
\end{coroll}

\newpage

\bibliographystyle{mystyle}
\bibliography{data}

\vspace{2em}
\noindent
\textit{Mikhail Kornev}\\
HSE University\\
Email: \texttt{mkorneff@mi-ras.ru}

\end{document}